\documentclass[11pt]{amsart}
\usepackage{amsmath,amssymb,amsthm,amsfonts}
\usepackage[alphabetic]{amsrefs}
\usepackage{mathrsfs}
\usepackage[arrow,matrix,curve,cmtip,ps]{xy}
\usepackage{graphicx,tikz}
\usepackage[hypertexnames=true, citecolor = green, colorlinks = false, linkcolor = red, linktoc = none, pdffitwindow = false, urlbordercolor = white]{hyperref} 
\usepackage[shortlabels]{enumitem}
\usepackage[final]{pdfpages}
\usepackage{thmtools}
\usepackage{cleveref}
\usepackage{multirow}

\allowdisplaybreaks

\newcommand{\IC}[0]{\mathbb{C}} 
 \newcommand{\IF}[0]{\mathbb{F}}

\newcommand{\CA}[0]{\mathcal{A}}

\newcommand{\CG}[0]{\mathcal{G}} 
 \newcommand{\CJ}[0]{\mathcal{J}}
 \newcommand{\CL}[0]{\mathcal{L}}
 
\newcommand{\CO}[0]{\mathcal{O}} 
\newcommand{\CQ}[0]{\mathcal{Q}}

 \newcommand{\FF}[0]{\mathfrak{F}}

\newcommand{\Aut}[0]{\operatorname{Aut}}

\newcommand{\gxp}{G \rtimes_\theta P}
\newcommand{\gpt}{G,P,\theta}
\newcommand{\pfin}{P^{(\text{fin})}}
\newcommand{\pinf}{P^{(\text{inf})}}

\newtheorem{theorem}{Theorem}[section]
\newtheorem{lemma}[theorem]{Lemma}
\newtheorem{proposition}[theorem]{Proposition}
\newtheorem{corollary}[theorem]{Corollary}
\newtheorem{notation}[theorem]{Notation}

\theoremstyle{remark}
\newtheorem{remark}[theorem]{\textnormal{\bfseries Remark}\bfseries}
\newtheorem{standingassumption}[theorem]{\textnormal{\bfseries Standing 
Assumption}\bfseries}

\newtheorem{definition}[theorem]{\textnormal{\bfseries Definition}\bfseries}
\newtheorem{example}[theorem]{\textnormal{\bfseries Example}\bfseries}
\newtheorem{examples}[theorem]{\textnormal{\bfseries Examples}\bfseries}

\newtheorem*{theorem*}{Theorem}
\newtheorem*{lemma*}{Lemma}
\newtheorem*{remark*}{Remark}

\numberwithin{equation}{section}

\newcommand{\Z}{\mathbb{Z}}
\newcommand{\N}{\mathbb{N}}

\makeatletter
\ifcsname phantomsection\endcsname
    \newcommand*{\qrr@gobblenexttocentry}[5]{}
\else
    \newcommand*{\qrr@gobblenexttocentry}[4]{}
\fi
\newcommand*{\addsubsection}{%
    \addtocontents{toc}{\protect\qrr@gobblenexttocentry}%
    \subsection}
\makeatother

\begin{document}
\title{The boundary quotient for algebraic dynamical systems}

\author{Nathan Brownlowe}
\address{School of Mathematics and Applied Statistics \\ University of Wollongong \\ \hspace*{4mm}Australia}
\email{nathanb@uow.edu.au}

\author{Nicolai Stammeier}
\address{Mathematisches Institut \\ Westf\"{a}lische Wilhelms-Universit\"{a}t M\"{u}nster \\ \newline\hspace*{4mm}Germany}
\email{n.stammeier@wwu.de}


\thanks{The second author was supported by ERC through AdG $267079$.}

\date{}

\subjclass[2010]{}


\begin{abstract}
We introduce the notion of accurate foundation sets and the accurate refinement property for right LCM semigroups. For right LCM semigroups with this property, we derive a more explicit presentation of the boundary quotient. In the context of algebraic dynamical systems, we also analyse finiteness properties of foundation sets which lead us to a very concrete presentation. Based on Starling's recent work, we provide sharp conditions on certain algebraic dynamical systems for pure infiniteness and simplicity of their boundary quotient.
\end{abstract}

\maketitle
\section*{Introduction}\label{intro}
\noindent All semigroups in this paper are assumed to be countable, discrete and left cancellative. Recall from \cite{BRRW} that a semigroup is \emph{right LCM} if the intersection of two principal right ideals is either empty or another principal right ideal. Examples of right LCM semigroups come from \emph{algebraic dynamical systems} $(\gpt)$, which consist of an action $\theta$ of a right LCM semigroup $P$ with identity by injective endomorphisms of a group $G$, subject to the condition that $pP \cap qP = rP$ implies $\theta_p(G) \cap \theta_q(G) = \theta_r(G)$ for all $p,q,r \in P$, see \cite{BLS2} for details and examples. It has been observed that the $C^*$-algebra $\CA[\gpt]$ associated to $(\gpt)$ in \cite{BLS2} is isomorphic to the full semigroup $C^*$-algebra of the right LCM semigroup $\gxp$, see \cite{BLS2}*{Theorem 4.4}. It is also know to be isomorphic to a Nica-Toeplitz algebra for a product system of right-Hilbert bimodules over the right LCM semigroup $P$, see \cite{BLS2}*{Theorem 7.9}. These two ways of viewing $\CA[\gpt]$ both indicate that this $C^*$-algebra tends to have proper ideals. Therefore, it is natural to search for a notion of a minimal quotient that is simple under reasonable assumptions on $(\gpt)$.

With regards to $C^*$-algebras of product systems of right-Hilbert bimodules, this quotient ought to be a Cuntz-Nica-Pimsner algebra. But so far only Nica covariance has been defined for product systems over right LCM semigroups, see \cite{BLS2}*{Definition 6.4}. Even worse, it does not seem to be clear what the general notion of Cuntz-Pimsner covariance for product systems over quasi-lattice ordered pairs should be, compare \cite{Fow2} and \cite{SY}. Recently, definitions for Cuntz-Pimsner covariance for product systems over Ore semigroups have been proposed in \cite{KS} and \cite{AM} which might lead to substantial progress in this direction. However, we remark that a right LCM semigroup can be far from satisfying the Ore condition.

There has been a successful attempt to identify the analogous quotient, called the \emph{boundary quotient}, for full semigroup $C^*$-algebras of right LCM semigroups with identity, see \cite{BRRW}. In fact, the authors also indicate how one could define this object for general semigroups, see \cite{BRRW}*{Remark 5.5}. Let us briefly review the idea behind this quotient, which goes back to \cite{CL2}: Firstly, recall from \cite{BLS1}*{Lemma 3.3} that the family of constructible right ideals $\CJ(S)$ for a right LCM semigroup with identity $S$ consists only of $\emptyset$ and the principal right ideals in $S$. A finite subset $F$ of $S$ is called a \emph{foundation set} if for every $s \in S$ there is $f \in F$ such that $sS \cap fS \neq \emptyset$. The boundary quotient $\CQ(S)$ of $C^*(S)$ is then obtained by imposing the additional relation $\prod_{s \in F} (1-e_{sS}) = 0$ for every foundation set $F$. It was shown in \cite{BRRW} that $\CQ(S)$ recovers classical objects such as $\CO_n$, provides an appealing perspective on Toeplitz and Cuntz-Pimsner algebras associated to self-similar actions, see \cite{BRRW}*{Subsection 6.4}, and may yield plenty of interesting new $C^*$-algebras related to Zappa-Sz\'{e}p products of monoids which had not been considered before.

As we know that $\gxp$ is right LCM for each algebraic dynamical system $(\gpt)$, the boundary quotient $\CQ(\gxp)$ deserves a closer examination. As it turns out, for most standard examples of such dynamics, the resulting right LCM semigroup $S=\gxp$ has two additional features: There are plenty of foundation sets $F$ such that $f_1S$ and $f_2S$ are disjoint for all distinct $f_1,f_2 \in F$. Such finite subsets $F$ will be called \emph{accurate foundation sets}. More importantly, every foundation set $F$ can be refined to an accurate foundation set $F_a$ in the sense that for every $f_a \in F_a$ there is $f \in F$ such that $f_a \in fS$. This feature will be named the \emph{accurate refinement property}, or property (AR) for short. If a right LCM semigroup $S$ has property (AR), then the defining relation 
\[\begin{array}{ll} 
\prod\limits_{f \in F} (1-e_{fS}) = 0 &\hspace*{4mm}\text{ for every foundation set } F\vspace*{2mm}\\
\multicolumn{2}{l}{\hspace*{-14mm}\text{can be replaced by the more familiar-looking relation}}\vspace*{2mm}\\
\sum\limits_{f \in F_a} e_{fS} = 1 &\hspace*{4mm}\text{ for every accurate foundation set } F_a, 
\end{array}\] 
see Proposition~\ref{prop:AR property -> familiar boundary quot pic}. We show that property (AR) is enjoyed by various types of known right LCM semigroups. 

If we are given additional information on $S$ in the sense that $S= \gxp$ for a (nontrivial) algebraic dynamical system $(\gpt)$, then we can say more about the structure of (accurate) foundation sets and hence about property (AR). This is the aim of Section~\ref{sec:ADS}, where we present a useful sufficient criterion on $(\gpt)$ for $\gxp$ to have property (AR), see Proposition~\ref{prop:refining FS to AFS for ADS}. As an application, we show that $\gxp$ has property (AR) provided that $P$ is directed or that incomparable elements in $P$ have disjoint principal right ideals, where we use $p \geq q :\Leftrightarrow p \in qP$, see Corollary~\ref{cor:AR for some gxp}. We note that these two options include the cases where $P$ is a group, an abelian semigroup, a free semigroup, or a Zappa-Sz\'{e}p product $X^* \bowtie G$ for some self-similar action $(G,X)$ as in \cite{BRRW}. In particular, the semigroups $\gxp$ arising from irreversible algebraic dynamical systems as defined in \cite{Sta1} have property (AR). To achieve Proposition~\ref{prop:refining FS to AFS for ADS} and hence the aforementioned results, we use a celebrated lemma of B.~H.~Neumann from \cite{Neu} about finite covers of groups by cosets of subgroups to conclude that it suffices to consider (accurate) foundation sets $F$ for $\gxp$ such that the index of $\theta_p(G)$ of $G$ is finite for all $(g,p) \in F$, see Proposition~\ref{prop:FS for gxp - pfin suffices}. 

Let $(\gpt)$ satisfy the assumptions of Proposition~\ref{prop:refining FS to AFS for ADS}, so that $\gxp$ has property (AR). If we combine the alternative presentation for $\CQ(\gxp)$ obtained in Proposition~\ref{prop:AR property -> familiar boundary quot pic} with the dynamic description $\CA[\gpt]$ of $C^*(\gxp)$, we arrive at a presentation of $\CQ(\gxp)$ which emphasises that it originates from a dynamical system, see Corollary~\ref{cor:identifying Q and O}. However, we observe that $\CQ(\gxp)$ may fail to admit a natural representation on $\ell^2(G)$: The representation exists if and only if $P$ is directed, see Proposition~\ref{prop:rep on G}. This is somewhat surprising as $\ell^2(G)$ is arguably a very natural state space for a dynamical system given by injective group endomorphisms of a group $G$. Nevertheless, we immediately get that the boundary quotient $\CQ(\gxp)$ is canonically isomorphic to the $C^*$-algebra $\CO[\gpt]$ studied in \cite{Sta1} for irreversible algebraic dynamical systems $(\gpt)$, see Corollary~\ref{cor:consistency for IADS}. Thus, one can regard the present paper as a continuation, and a vast generalisation of essential parts from \cite{Sta1}, though the employed techniques are quite different. 

The topic we have not addressed so far is simplicity and pure infiniteness of $\CQ(\gxp)$. In \cite{Sta1}, the author showed that $\CO[\gpt]$ is purely infinite and simple provided a certain amenability condition and $\bigcap_{p \in P} \theta_p(G) = \{1\}$ hold, see \cite{Sta1}*{Theorem 3.26}. But it remained unclear whether these sufficient conditions where also necessary for irreversible algebraic dynamical systems. They were known to be sharp for the case where $G$ is abelian and $G/\theta_p(G)$ is finite for all $p \in P$ by \cite{Sta2}*{Corollary 5.10}.

Fortunately, Starling has recently applied deep results from \cite{EP} and \cite{BOFS} precisely to boundary quotients of right LCM semigroups to obtain a characterisation of simplicity, see \cite{Star}*{Theorem 4.12}. We analyse his conditions in the context of algebraic dynamical systems in order to express them directly in terms of $(\gpt)$. This leads to much more explicit conditions in important special cases, see Corollary~\ref{cor:applying Starling's result to reg ADS with trivial P*}. Mostly, we restrict our attention to the case where $P$ is right cancellative, simply because we lack examples for algebraic dynamical systems with a right LCM $P$ that is not right cancellative. Regarding simplicity of $\CO[\gpt]$ for irreversible algebraic dynamical systems, we now achieve a proper characterisation, see Corollary~\ref{cor:simplicity for IADS}, and the conditions turn out to be slightly milder than in \cite{Sta1}. Finally, we address classifiability of $\CQ(\gxp)$ in Theorem~\ref{thm:UCT Kirchberg algs}.

The paper is organised as follows: In Section~\ref{sec:background} we recall the notions of the boundary quotient and the inverse semigroup of a right LCM semigroup as well as the key result from \cite{Star} concerning simplicity. Accurate foundation sets and property (AR) are introduced and studied for certain right LCM semigroups in Section~\ref{sec:AFS}. In Section~\ref{sec:ADS} we focus on establishing property (AR) for right LCM semigroups constructed from algebraic dynamical systems and analyse finiteness properties of (accurate) foundation sets. In the final Section~\ref{sec: BQ for ADS}, we start off with some observations concerning basic structural properties of the boundary quotient for algebraic dynamical systems $(\gpt)$, before we discuss simplicity and pure infiniteness.


\section{Background}\label{sec:background}
\noindent In this section we give the necessary background on semigroups and their 
$C^*$-algebras, including the full semigroup $C^*$-algebra $C^*(S)$, and its 
boundary quotient $\CQ(S)$. In the second subsection we discuss Starling's 
results from \cite{Star}, where he studied the boundary quotient of right 
LCM semigroups using an inverse semigroup (and groupoid) approach.

\subsection{The boundary quotient for right LCM 
semigroups}\label{subsec:boundary quots}~\\
\noindent Within this section, we briefly recall the construction of $C^*(S)$ from \cite{Li1} and the notion of the boundary quotient $\CQ(S)$ of $C^*(S)$ for right LCM semigroups from \cite{BRRW}*{Definition 5.1}.

In \cite{Li1}, the full semigroup $C^*$-algebra $C^*(S)$ of a discrete and left cancellative semigroup $S$ is defined using additional relations for projections $e_X$ arising from right ideals $X$ in $S$ that are part of the \emph{family of constructible right ideals} $\CJ(S)$. This is the smallest family of right ideals of $S$ satisfying
\begin{enumerate}[(a)]
 \item $S, \emptyset \in \CJ(S)$ and
 \item $X \in \CJ(S)$ and $s \in S$ implies $sX,s^{-1} X \in \CJ(S)$.
\end{enumerate}
The general form of a constructible right ideal is given in \cite{Li1}*{Equation~(5)}. We note that $\CJ(S)$ is also closed under finite intersections, a fact that can be derived from (a) and (b) using $sS \cap tS = s(s^{-1}(tS))$. 

\begin{definition}\label{def: Li's full algebra}
Let $S$ be a discrete left cancellative semigroup. The {\em full semigroup
$C^*$-algebra} $C^*(S)$ is the universal $C^*$-algebra generated by isometries
$(v_s)_{s\in S}$ and projections $(e_X)_{X\in\CJ(S)}$ satisfying
\[\begin{array}{llcll}
\text{(L1)} & v_sv_t=v_{st}, && \text{(L2)} & v_se_Xv_s^*=e_{sX},\vspace*{2mm}\\
\text{(L3)} & e_S=1, e_\emptyset=0, &\text{ and }& \text{(L4)} & e_Xe_Y=e_{X\cap Y},
\end{array}\]
for all $s,t\in S$, $X,Y\in\CJ(S)$.
\end{definition}

\noindent Note that (L2) and (L3) give $v_sv_s^*=e_{sS}$ for all $p\in S$. If $S$ is a right LCM semigroup with identity, then $\CJ(S) = \{sS \mid s \in S\} \cup \{\emptyset\}$, see \cite{BLS1}*{Lemma 3.3}. From now on, let $S$ be a right LCM semigroup with identity.

\begin{definition}\label{def:foundation sets}
A finite subset $F \subset S$ is called a \emph{foundation set} for $S$ if, for every $s \in S$, there exists $t \in F$ satisfying $sS \cap tS \neq \emptyset$. The collection of foundation sets for $S$ is denoted by $\FF(S)$.
\end{definition}

\begin{remark}\label{rem:foundation sets}
We note the following simple observations:
\begin{enumerate}[(a)]
\item If $S$ is directed, then every finite subset of $S$ is a foundation set.
\item $F \subset S$ is a foundation set if and only if it is finite and $sS \cap \bigcup_{t \in F} tS \neq \emptyset$ for all $s \in S$. Since $S$ is right LCM, this means that for each principal right ideal $sS$, there is $s' \in S$ such that $ss' \in \bigcup_{t \in F}tS$. So this union can be thought of as a cofinal subset of $S$ with respect to the partial order on $S$ induced by reverse inclusion of associated principal right ideals.
\end{enumerate}
\end{remark}

\begin{definition}\label{def:boundary quotient}
The \emph{boundary quotient} $\CQ(S)$ is the quotient of $C^*(S)$ by
\begin{equation}\label{eq:Q condition}
\begin{array}{l}\prod\limits_{s \in F}(1-e_{sS}) = 0 \hspace*{6mm}\text{for every foundation set } F.\end{array}\tag{Q}
\end{equation}
\end{definition}

\noindent We shall denote the images of the isometries $v_s$ and the projections $e_{sS}$ for $s \in S$ under the quotient map by $\bar{v}_s$ and $\bar{e}_{sS}$, respectively.

We point out that \eqref{eq:Q condition} has the flavour of the summation relation used for $\CO_n, 2 \leq n < \infty$. This is the essence of Proposition~\ref{prop:AR property -> familiar boundary quot pic}. 

\subsection{The inverse semigroup approach}\label{subsec: inv semigroup 
approach to BQ}~\\
In \cite{Star} Starling uses techniques and machinery from inverse semigroups 
and groupoids to study the boundary quotient $\CQ(S)$ of a right LCM semigroup 
$S$. In particular, he applies the 
machinery from \cite{EP} and the results of \cite{BOFS}. In this section we 
recall the construction of an inverse semigroup $I(S)$ for a discrete, left 
cancellative semigroup $S$, and then some of the terminology, 
notation and results from \cite{Star}.

\begin{definition}\label{def:inv sgp for left canc sgp}
For a discrete, left cancellative semigroup $S$ let $I(S)$ be the multiplicative 
subsemigroup of $C^*(S)$ generated by $0$ and $v_s,v_s^*$ for $s \in 
S$. The set of idempotents in $I(S)$ is denoted by $E(S)$.
\end{definition}

\begin{lemma}\label{lem:I(S) is an inv sgp}
$I(S)$ is an inverse semigroup with identity and zero. $E(S)$ is given by $\{ 
e_X \mid X \in \CJ(S)\}$, where $\CJ(S)$ denotes the family of constructible 
right ideals in $S$. If $S$ is right LCM, then $I(S)$ equals $\{0\} \cup \{ 
v_sv_t^* \mid s,t \in S\}$ and $E(S) = \{0\} \cup \{e_{sS} \mid s \in S\}$.
\end{lemma}
\begin{proof}
The first claim is straightforward. If we consider an arbitrary finite product 
$v_{s_1}^*v_{s_2}\cdots v_{s_{n-1}}^*v_{s_n}$, then its range projection
\[v_{s_1}^*v_{s_2}\cdots v_{s_{n-1}}^*v_{s_n}(v_{s_1}^*v_{s_2}\cdots 
v_{s_{n-1}}^*v_{s_n})^* = e_{s_1^{-1}(s_2(\dots (s_{n-1}^{-1}(s_nS)\dots )}\]
is the projection corresponding to the constructible right ideal 
\[s_1^{-1}(s_2(\dots (s_{n-1}^{-1}(s_nS)\dots ) \in \CJ(S).\] 
Hence we get $E(S) = \{ e_X \mid X \in \CJ(S)\}$. Now suppose $S$ is right 
LCM. Then 
$v_s^*v_t = v_s^*e_{sS \cap tS}v_t$ vanishes unless $sS \cap tS = rS, 
r=ss'=tt'$ for some $r,s',t' \in S$, in which case we get $v_s^*v_t = 
v_{s'}v_{t'}^*$. Finally, we know that $\CJ(S) = \{\emptyset\} \cup \{sS \mid 
s \in S\}$ for right LCM semigroups. 
\end{proof}

\noindent $I(S)$ can also be defined via partial bijections $\Lambda_s:S \to 
sS$ and their partial inverses $\Lambda_s^*:sS \to S$ since $S$ is left 
cancellative. 

To every inverse semigroup $I$, we can associate the \emph{tight groupoid} 
$\CG_{\text{tight}}(I)$, which is a groupoid of germs for a certain action of 
the inverse semigroup on a particular spectrum, see \cite{Exe3} for details. For $I(S)$ of a right LCM semigroup $S$, the boundary quotient 
$\CQ(S)$ is isomorphic to $C^*_{\text{tight}}(I(S)) \cong 
C^*(\CG_{\text{tight}}(I(S)))$, see \cite{Star}*{Theorem 3.7 and Subsection 4.1}.

We need two more concepts before we can state Starling's main theorem on 
simplicity of $\CQ(S)$ for right LCM $S$. Following the 
convention of 
\cite{Star}, we denote
\begin{equation}\label{eq: square brackets}
[s,t]:=v_sv_t^*. 
\end{equation}
Note that $[s,t] = 
[s',t']$ holds if and only if we have $s'=sx$ and $t'=tx$ for some $x \in S^*$.

The following is
\cite{Star}*{Definition 4.6}, which is inspired by the work of Crisp and Laca, 
see \cite{CL2}*{Definition 5.4}.

\begin{definition}\label{def:core for right LCM}
For a right LCM semigroup $S$ with identity, the \emph{core} of $S$ is the set 
$S_0 := \{ s \in S \mid sS \cap tS \neq \emptyset \text{ for all } t \in S\}$. 
\end{definition}

\noindent It is immediate that the group of invertible elements $S^*$ in $S$ is a subset of the core $S_0$.

We now state Starling's \cite{Star}*{Theorem~4.12}, although we have not yet 
written down conditions \eqref{eq:condition H} and \eqref{eq:condition EP}; we 
will do so after the statement.

\begin{theorem}\label{thm:starling simple}
Let $S$ be a right LCM semigroup with identity which satisfies 
\eqref{eq:condition H}. Then $\CQ(S)$ is simple if and only if
\begin{enumerate}
\item[\textnormal{(1)}] $\CQ(S) \cong C^*_r(\CG_{\text{tight}}(I(S)))$, and 
\item[\textnormal{(2)}] for all $s,t \in S_0$, the element $[s,t]$ satisfies 
\eqref{eq:condition EP}.
\end{enumerate}
\end{theorem}

\noindent Let us explain the conditions \eqref{eq:condition H} and \eqref{eq:condition EP}. Condition \eqref{eq:condition H} characterises Hausdorffness of the tight groupoid of $I(S)$, see \cite{Star}*{Proposition 4.1}:

\begin{equation}\label{eq:condition H}
 \begin{array}{l} \text{For all $s,t \in S$ with $sS \cap tS \neq \emptyset$, 
 there is a finite subset $F \subset S$}\\ 
 \text{with $sf=tf$ for all $f \in F$ such that the following holds: If $r \in 
 S$}\\
 \text{satisfies $sr=tr$, then there exists $f \in F$ with $rS \cap fS \neq 
 \emptyset$.}
 \end{array}\tag{H}
\end{equation}

\begin{remark}\label{rem:condition H general}
If we have $s=t$, then we can simply choose $F=\{1\}$ . Now if $S$ is right 
cancellative, then $sr=tr$ implies $s=t$. Hence \eqref{eq:condition H} holds 
whenever $S$ is right cancellative. 
\end{remark}

\noindent To present condition \eqref{eq:condition EP}, we need to define the notion of weakly fixed idempotents in inverse semigroups, see \cite{Star}*{Definition 4.8}.

\begin{definition}\label{def:weakly fixed}
Let $I$ be an inverse semigroup, $a \in I$ and $e \in E(I)$ such that $a^*a 
\geq e$. The idempotent $e$ is said to be \emph{weakly fixed} by $a$ if 
$afa^*f \neq 0$ for all $f \in E(I)\setminus \{0\}$ with $f \leq e$.
\end{definition}

\noindent Since we are interested in inverse semigroups built from right LCM 
semigroups with identity, let us recall the conclusion of \cite{Star}*{Lemma 
4.9}:

\begin{lemma}\label{lem:weakly fixed for right LCM}
Let $S$ be a right LCM semigroup with identity and $I(S)$ the associated 
inverse semigroup. $[s,t]$ fixes $[tt',tt']$ weakly if and only if $st'rS \cap 
tt'rS \neq \emptyset$ for all $r \in S$. 
\end{lemma}

\noindent As stated in \cite{Star}*{Lemma 4.11}, condition \eqref{eq:condition 
EP} for $[s,t] \in I(S)$ is given by:
\begin{equation}\label{eq:condition EP}
\begin{array}{l} \text{Whenever $[s,t]$ fixes $[tt',tt']$ weakly, there is a foundation set $F$}\\
\text{such that $st'f=tt'f$ for all  $f \in F$.}
\end{array}\tag{EP}
\end{equation}

\begin{remark}\label{rem:EP for special cases}
There are some special cases:
\begin{enumerate}[(a)]
\item If there is no $t' \in S$ such that $[s,t]$ fixes $[tt',tt']$ weakly, then \eqref{eq:condition EP} holds for $[s,t]$.
\item For $s=t$ the foundation set $F= \{1\}$ gives \eqref{eq:condition EP} for $[s,t]$. 
\item Suppose $S$ is right cancellative and $[s,t] \in I(S)$ fixes some $[tt',tt']$ weakly. If $[s,t]$ satisfies \eqref{eq:condition EP}, then $s$ and $t$ have to be the same.
\end{enumerate}
 \end{remark}

\section{Foundation sets made accurate}\label{sec:AFS}	
\noindent Throughout this section let $S$ be a right LCM semigroup with identity. We will now introduce accurate foundation sets and accurate refinements of foundation sets. These lead to a clearer picture of the boundary quotient $\CQ(S)$ provided that accurate refinements are always possible. This feature of $S$ is called the accurate refinement property, or property (AR) for short. We show that many known right LCM semigroups have property (AR). In fact, we are not aware of an example of a right LCM semigroup that does not have property (AR).

\begin{definition}\label{def:AFS}
$F \in \FF(S)$ is called an \emph{accurate foundation set} for $S$ if $sS \cap tS = \emptyset$ holds for all $s,t \in F, s \neq t$. The collection of accurate foundation sets is denoted by $\FF_a(S)$.
\end{definition}

\begin{remark}\label{rem:AFS for directed right LCM}
If $S$ is directed, then $\FF(S)$ consists of all finite subsets of $S$, see Remark~\ref{rem:foundation sets}~(a), and it is apparent that $\FF_a(S)$ is equal to $S$ as a set. For general right LCM $S$, accurate foundation sets consisting of a single point correspond to elements of the core $S_0$ of $S$.
\end{remark}

\begin{definition}\label{def:AR property}
$S$ is said to have the \emph{accurate refinement property}, or property (AR) for short, if, for all $F \in \FF(S)$, there exists $F_a \in \FF_a(S)$ such that $F_a \subset FS$. This means that for every $f_a \in F_a$, there is a $f \in F$ with $f_a \in fS$.
\end{definition}

\begin{proposition}\label{prop:AR property -> familiar boundary quot pic}
$\sum_{s \in F_a} \bar{e}_{sS} = 1$ holds for all accurate foundation sets $F_a$ for $S$. If $S$ has property \textnormal{(AR)}, then $\CQ(S)$ is the quotient of $C^*(S)$ by the relation 
\begin{equation}\label{eq:Q_a condition}
\begin{array}{l}\sum\limits_{s \in F} e_{sS} = 1 \hspace*{6mm}\textnormal{ for every accurate foundation set } F.\end{array}\tag{$\text{Q}_a$}
\end{equation}
In other words, $\CQ(S)$ is the universal $C^*$-algebra generated by a representation of $S$ by isometries $\bar{v}_s$ and projections $(\bar{e}_{X})_{X \in \CJ(S)}$ subject to the relations \textnormal{(L1) -- (L4)}, and \eqref{eq:Q_a condition}.
\end{proposition}
\begin{proof}
Let $F_a$ be an accurate foundation set for $S$, that is 
\begin{enumerate}[1)]
\item $F_a$ is a foundation set and
\item $sS \cap tS = \emptyset$ for all distinct $s,t \in F_a$.
\end{enumerate}
1) implies $\prod_{s \in F_a} (1-\bar{e}_{sS}) = 0$, see Definition~\ref{def:boundary quotient}. Now (L4) and 2) yield $\prod_{s \in F_a} (1-\bar{e}_{sS}) = 1 - \sum_{s \in F_a} \bar{e}_{sS}$, so we get \eqref{eq:Q_a condition}.

Now let $F \in \FF(S)$. We need to show that \eqref{eq:Q condition} holds for $F$ by using the structure of $C^*(S)$ and \eqref{eq:Q_a condition}. If $S$ has property (AR), then there is $F_a \in \FF_a(S)$ which refines $F$, that is, for each $s \in F_a$, there exists $t \in F$ with $s \in tS$. On the level of $C^*(S)$, this implies $1 - e_{tS} \leq 1-e_{sS}$. Since all these projections commute, we get 
\[\begin{array}{l} 0 \leq \prod\limits_{t \in F}(1-e_{tS}) \leq \prod\limits_{s \in F_a} (1-e_{sS}) = 1 - \sum\limits_{s \in F_a}e_{sS}. \end{array}\] 
But the right hand side vanishes once we impose relation \eqref{eq:Q_a condition} and hence \eqref{eq:Q condition} holds for $F$. This shows that \eqref{eq:Q condition} and \eqref{eq:Q_a condition} are equivalent relations provided that $S$ has property (AR).
\end{proof}


\noindent Similar presentations of $\CQ(S)$ for right LCM semigroups with property (AR) have been obtained in special cases, see for instance \cite{LR3}*{Corollary 6.2} or \cite{BRRW}*{Subsection 6.4}. We will now show, that these examples and many more right LCM semigroups have property (AR). 

Recall that reverse inclusion of principal right ideal defines a partial order on $S$, i.e. $s \leq t$ if $t \in sS$ for $s,t \in S$. If $s \leq t$ or $s \geq t$ holds, then $s,t \in S$ are said to be \emph{comparable}. 

\begin{remark}\label{rem:incomp elts -> disj p r ideal -> right LCM}
Every left cancellative semigroup with the property that incomparable elements have disjoint principal right ideals is right LCM.
\end{remark}

\begin{proposition}\label{prop:AR for some right LCM sgps}
Suppose that 
\begin{enumerate}
\item[\textnormal{(1)}] $S$ is directed with respect to $\leq$, or
\item[\textnormal{(2)}] incomparable elements have disjoint principal right ideals.
\end{enumerate}
Then every $F \in \FF(S)$ has an accurate refinement $F_a \in \FF_a(S)$ satisfying $F_a \subset F$. In particular, $S$ has property \textnormal{(AR)}.
\end{proposition}
\begin{proof}
Suppose first that $S$ is directed and let $F \in \FF(S)$. Then $F \neq \emptyset$ and every $p \in F$ yields an accurate refinement $F_a:= \{p\} \in \FF_a(S)$ for $F$. Now let $S$ satisfy (2) and $F \in \FF(S)$. If there are $p,q \in F$ with $p \neq q$ and $pS \cap qS \neq \emptyset$, (2) implies that $p \in qP$ or $q \in pP$. If $p \in qP$, then $F' := F \setminus \{p\} \in \FF(S)$, and otherwise we get $F':= F \setminus \{q\} \in \FF(S)$. Hence we can remove redundant elements from $F$ until there are only those left that correspond to mutually disjoint right ideals and the output is an accurate refinement of $F$.
\end{proof}

\noindent The class of right LCM semigroups to which Proposition~\ref{prop:AR for some right LCM sgps} applies is large and we list a number special cases to demonstrate this.

\begin{corollary}\label{cor:AR for some right LCM sgps}
If $S$ is
\begin{enumerate}
\item[\textnormal{(3)}] a group,
\item[\textnormal{(4)}] abelian,
\item[\textnormal{(5)}] isomorphic to $\IF_n^+$ for some $1 \leq n < \infty$, or 
\item[\textnormal{(6)}] given by $X^*\bowtie G$ for a self-similar action $(G,X)$,
\end{enumerate}
then either \textnormal{(1)} or \textnormal{(2)} holds. In particular, $S$ has property \textnormal{(AR)}.
\end{corollary}
\begin{proof}
(3) and (4) both imply (1), so $S$ has property (AR) by Proposition~\ref{prop:AR for some right LCM sgps}. (5) is a special case of (6). Due to \cite{BRRW}*{Theorem 3.8}, (6) forces (2) and hence Proposition~\ref{prop:AR for some right LCM sgps} shows property (AR).
\end{proof}

\begin{remark}\label{rem:AR for inf gen free sgp}
$S=\IF_\infty^+$ also has property (AR), but for trivial reasons since any foundation set $F$ for $S$ has to contain the identity of $S$. For completeness, we note that $\FF_a(S) = \bigl\{\{1\}\bigr\}$. 
\end{remark}

\begin{remark}\label{rem:Lawson's left Rees monoids}
In \cite{Law}, Lawson considered so-called \emph{left Rees monoids}, which are left cancellative semigroups with identity that satisfy condition (2) from Proposition~\ref{prop:AR for some right LCM sgps} and the ascending chain condition for principal right ideals. The last condition means that every principal right ideal is properly contained in only a finite number of principal right ideals. By Proposition~\ref{prop:AR for some right LCM sgps}, all left Rees monoids have property (AR). 

According to \cite{Law}*{Theorem 3.7}, attributed to Perrot \cite{Per}, left Rees monoids can be characterised as Zappa-Sz\'{e}p products of free monoids by groups. Moreover, new examples of left Rees monoids can be constructed out of known ones, see \cite{Law}*{Section 4} for details. 

Finally, let us mention that \cite{Law}*{Examples 2.8} provides a number of interesting examples of left Rees monoids. \cite{Law}*{Examples 2.8 (iv)} might be particularly interesting because a left Rees monoid is constructed from an arbitrary left cancellative semigroup using Rhodes expansions.
\end{remark}

\noindent From what we have gathered so far, it seems feasible to explore property (AR) for other kinds of Zappa-Sz\'{e}p products $U \bowtie A$. Indeed, \cite{BRRW}*{Lemma 3.3} provides a sufficient criterion for $U \bowtie A$ to be a right LCM semigroup. More importantly, \cite{BRRW}*{Remark 3.4} explains that, given the requirements of \cite{BRRW}*{Lemma 3.3}, the structure of $\CJ(U \bowtie A)$ is governed by $\CJ(U)$, i.e.,
\begin{equation}\label{eq:J(U bowtie A) vs J(U)}
(u,a)U\bowtie A \cap (v,b)U\bowtie A \neq \emptyset \Longleftrightarrow uU \cap vU \neq \emptyset
\end{equation} 
for all $(u,a),(v,b) \in U\bowtie A$. The proof of \cite{BRRW}*{Lemma 3.3} actually shows that if $w$ is a right LCM for $u$ and $v$ in $U$, then there is $c \in A$ such that
\begin{equation}\label{eq:right LCM for U bowtie A by right LCM for U}
(u,a)U\bowtie A \cap (v,b)U\bowtie A = (w,c)U\bowtie A.
\end{equation}

\begin{proposition}\label{prop:AR for U bowtie A iff AR for U}
Suppose $S = U \bowtie A$ is such that $U$ is right LCM, $\CJ(A)$ is totally ordered by inclusion and $U \to U, u \mapsto a \cdot u$ is bijective for every $a \in A$. Then $S$ is a right LCM semigroup with identity and $S$ has property (AR) if and only if $U$ has property (AR). 
\end{proposition}
\begin{proof}
$S$ is right LCM by \cite{BRRW}*{Lemma 3.3}. For each $E \subset U$ and every family $(a_u)_{u \in E}$ we let $F(E,(a_u)) := \{ (u,a_u) \mid u \in E\}$. Similarly, given $F \subset S$, we set $E(F) := \{ u \mid (u,a) \in F\}$. By \eqref{eq:J(U bowtie A) vs J(U)}, we have $E \in \FF(U)$ if and only if $F(E,(a_u)) \in \FF(S)$, and moreover, $E$ is accurate if and only if $F(E,(a_u))$ is accurate (for every family $(a_u)_{u \in E}$). Likewise, \eqref{eq:J(U bowtie A) vs J(U)} implies that $F \in \FF(S)$ holds if and only if $E(F) \in \FF(U)$. In addition, accuracy of $F$ is equivalent to $E(F)$ being accurate.

Now suppose $U$ has property (AR). Starting with $F \in \FF(S)$, we can refine $E(F) \in \FF(U)$ to some accurate foundation set $E(F)_a$. Take $u \in E(F)_a$. Since $E(F)_a$ is an accurate refinement for $E(F)$, there is $(v,b) \in F$ such that $u \in vU$. By \eqref{eq:right LCM for U bowtie A by right LCM for U}, there is $a_u \in A$ satisfying $(u,a_u) \in (v,b)S$. It follows that $F(E(F)_a,(a_u))$ is an accurate refinement of $F$ because $E(F)_a$ is accurate.

Conversely, assume that $S$ has property (AR). If $E \in \FF(U)$, then we know that $F(E,(a_u)) \in \FF(S)$ (for every family $(a_u)_{u \in E}$) and we can refine this foundation set by some $F_a \in \FF_a(S)$. By construction, $E(F_a)$ is an accurate refinement of $E$.
\end{proof}

\begin{examples}\label{ex:bowtie with AR}
We have already seen that the Zappa-Sz\'{e}p product $X^*\bowtie G$ associated to a self-similar action $(G,X)$ has property (AR). In fact, we can use Proposition~\ref{prop:AR for U bowtie A iff AR for U} to see that all of the examples of right LCM Zappa-Sz\'{e}p products in \cite{BRRW}*{Section~3} have property (AR).
\begin{enumerate}[(a)]
\item For $m$ and $n$ positive integers the positive cone $\text{BS}(m,n)^+$ of the Baumslag-Solitar group $\text{BS}(m,n)=\langle a,b|ab^m=b^na\rangle$ is a Zappa-Sz\'{e}p product of the form $\IF_n^+\bowtie \N$. Since we know from Corollary~\ref{cor:AR for some right LCM sgps} that $\IF_n^+$ has property (AR), Proposition~\ref{prop:AR for U bowtie A iff AR for U} says that each $\text{BS}(m,n)^+$ has property (AR).

\item The semigroups $\N\rtimes\N^\times$ and $\Z\rtimes\Z^\times$ can be described as Zappa-Sz\'{e}p products $U\bowtie A$ with $U=\{(r,x):x\ge 1,\, 0\le r<x\}$. To see that $U$ has property (AR), suppose $F=\{(r_1,x_1),\dots, (r_n,x_n)\}$ is a foundation set. Then for $x$ the least common multiple of $x_1,\dots,x_n$ the set $F_a=\{(0,x),\dots,(x-1,x)\}$ is an accurate foundation set. Morover, because of the structure of the principal right ideals and since $F$ is a foundation set, for each $(r,x)\in F_a$ the ideal $(r,x)U$ must be contained in one of $(r_i,x_i)U$. So $F_a$ is an accurate refinement of $F$, and hence $U$ has property (AR). Proposition~\ref{prop:AR for U bowtie A iff AR for U} now says that $\N\rtimes\N^\times$ and $\Z\rtimes\Z^\times$ both have property (AR).

\item If $G$ acts self-similarly on two alphabets $X$ and $Y$, and there is a bijection $\theta:Y\times X\to X\times Y$ such that the conditions given in \cite{BRRW}*{Proposition~3.10} hold, then there is a natural Zappa-Sz\'{e}p product $\IF_\theta^+\bowtie G$, where the semigroup $\IF_\theta^+$ is a $2$-graph with a single vertex. In general $\IF_\theta^+$ is not right LCM, but, for instance, it is right LCM when the sizes of $X$ and $Y$ are coprime, and $G=\Z$ acts as an odometer on both $X$ and $Y$. In this case, and if $X$ has size $m$ and $Y$ 
has size $n$, then $\IF_\theta^+$ is isomorphic to the subsemigroup of $U$ from (ii) generated by $(0,m),\dots,(m-1,m),(0,n),\dots,(n-1,n)$. The arguments above in (b) apply, and hence $\IF_\theta^+$ has property (AR). 
Proposition~\ref{prop:AR for U bowtie A iff AR for U} now says that the product of two (coprime) odometer actions $\IF_\theta^+\bowtie\Z$ has property (AR).   
\end{enumerate}
\end{examples}

\noindent Example~\ref{ex:bowtie with AR}~(b) can also be viewed as an elementary example of a semigroup built from an algebraic dynamical system $(\gpt)$ as $S=\gxp$. The natural question whether property (AR) passes from $P$ to $S$ under suitable conditions requires some preparation and will be examined in the first part of the next section.

\section{Foundation sets for algebraic dynamical systems}\label{sec:ADS}

Recall from \cite{BLS2}*{Definition 2.1} that an \emph{algebraic dynamical system} 
$(\gpt)$ is an action $\theta$ of a right LCM semigroup with identity $P$ by 
injective endomorphisms of a group $G$, subject to the condition that $pP \cap 
qP = rP$ implies $\theta_p(g) \cap \theta_q(G) = \theta_r(G)$ for all $p,q,r 
\in P$. 
In this section we aim to establish property (AR) for a large class of right LCM semigroups $S$ built from algebraic dynamical systems $(\gpt)$.

From now on let $(\gpt)$ denote an algebraic dynamical system. In addition, let $\pfin$ denote the subsemigroup of $P$ consisting of those $p \in P$ for which $G/\theta_p(G)$ is finite. For convenience, we shall usually denote $\gxp$ by $S$ within this section. 

Let us remind ourselves of the structure of $\CJ(S)$ as described in \cite{BLS2}*{Proposition 4.2} since this will be essential.

\begin{lemma}\label{lem:ideal intersection for gxp}
For all $(g,p),(h,q) \in S$, we have
\[(g,p)S \cap (h,q)S = \begin{cases}
(g\theta_p(k),r)S &\text{if there are } r \in P \text{ and } k \in G \text{ with}\\ 
&pP \cap qP = rP 
\text{ and } g\theta_p(k) \in h\theta_q(G), \\
\emptyset &\text{otherwise.}
\end{cases}\]
\end{lemma} 

\noindent Recall that $p \geq q$ is the same as $p \in qP$.

\begin{lemma}\label{lem:P_F construction for fin subset of gxp}
Given a finite subset $F \subset S$, there exists a finite set $P_F \subset P$ with the following properties:
\begin{enumerate}
\item[\textnormal{(i)}] Whenever $p \in P$ and $(h,q) \in F$ satisfy $pP \cap qP \neq \emptyset$, there is $q' \in P_F$ such that $pP \cap q'P \neq \emptyset$. 
\item[\textnormal{(ii)}] For each $q \in P_F$ there exists $p \in P$ such that $pP \cap qP \neq \emptyset$ and $pP \cap q'P = \emptyset$ for all $q' \in P_F$ with $q' \not\leq q$.
\item[\textnormal{(iii)}] For each $q \in P_F$ there exists $(h',q') \in F$ such that $q \geq q'$.
\end{enumerate} 
\end{lemma}
\begin{proof}
Let $F_1 \subset P$ be a complete set of representatives for 
\[\begin{array}{l}\bigl\{ \bigcap\limits_{(h,q) \in F'} qP \mid F' \subset F\bigr\} \setminus \{\emptyset\} \subset \CJ(P).\end{array}\]
Pick $q_1 \in F_1$ which is minimal in the sense that $q_1 \geq q$ implies $q_1P=qP$ for all $q \in F_1$. Let $F_1':= \{q \in F_1 \mid qP = q_1P \}$ and $E_0:= \emptyset$. If there is $p \in P$ such that $pP \cap q_1P \neq \emptyset$ whereas $pP \cap qP = \emptyset$ for all $q \in (F_1 \setminus F_1') \cup \{ q' \in E_0 \mid q' \not\leq q_1\}$, then we say that $q_1$ is indispensable and set $E_1:= E_0 \cup \{q_1\}$. If $q_1$ is dispensable, we choose $E_1 = E_0$ and note that by construction of $F_1$ we have: Whenever $q_1P \cap pP \neq \emptyset$ for some $p \in P$, there exists $q \in F_1 \setminus F_1'$ with $q \geq q_1$ such that $qP \cap pP \neq \emptyset$. 

Next, we define $F_2 := F_1 \setminus F_1'$ and repeat the procedure for some $q_2 \in F_2$ which is minimal in the sense that $q_2 \geq q$ implies $q_2P=qP$ for all $q \in F_2$.  Let $F_2':= \{q \in F_2 \mid qP = q_2P \}$. If there is $p \in P$ such that $pP \cap q_2P \neq \emptyset$ whereas $pP \cap qP = \emptyset$ for all $(h,q) \in (F_2 \setminus F_2') \cup \{ q' \in E_1 \mid q' \not\leq q_2\}$, then we set $E_2:= E_1 \cup \{q_2\}$. Otherwise we take $E_2:= E_1$. Finally, setting $F_3 := F_2 \setminus F_2'$ allows us to iterate this procedure. After finitely many steps, we get a finite set $E_n=:P_F$ which satisfies (i)--(iii) because it is a minimal subset of indispensable elements of $F_1$.
\end{proof}

\noindent It is clear from the construction that $P_F$ is non-empty if and only if $F$ is. If $P$ is directed, it is easy to see that $P_F$ consists of a single element $p_F$ with $\bigcap_{(h,q) \in F} qP = p_FP$.

\begin{lemma}\label{lem:char FS for gxp by P_F}
A finite subset $F$ of $S$ is a foundation set for $S$ if and only if there exists a foundation set $P_F$ for $P$ such that
\begin{equation}\label{eq:P_F cover of G}
\bigcup_{\substack{(h',q') \in F:\\ q' \leq q}} h'\theta_{q'}(G) = G \text{ holds for all } q \in P_F.
\end{equation}
\end{lemma}
\begin{proof}
Suppose $F$ is a foundation set and $P_F \subset P$ is obtained via Lemma~\ref{lem:P_F construction for fin subset of gxp}. So for every $(g,p) \in S$, there exists $(h,q) \in F$ such that $(g,p)S \cap (h,q)S \neq \emptyset$. According to Lemma~\ref{lem:ideal intersection for gxp}, this implies $pP \cap qP \neq \emptyset$. By condition (i) for $P_F$ from Lemma~\ref{lem:P_F construction for fin subset of gxp}, there is $q' \in P_F$ satisfying $pP \cap q'P \neq \emptyset$, so we get $P_F \in \FF(P)$. Concerning \eqref{eq:P_F cover of G}, we note that it suffices to prove this for all minimal elements of $P_F$. But if $q \in P_F$ is minimal among the elements of $P_F$, then (ii) implies that there exists $p \in P$ such that $pP \cap qP \neq \emptyset$ whereas $pP \cap q'P = \emptyset$ for all $q' \in P_F, q' \neq q$. Without loss of generality, we can assume that $p$ belongs to $qP$ since we may replace it with $p' \in qP$ satisfying $pP \cap qP = p'P$. Let $g \in G$. Since $F$ is a foundation set for $S$ there is $(h',q') \in F$ such that $(g,p)S \cap (h',q')S \neq \emptyset$. In particular, we get $(g,q)S \cap (h',q')S \neq \emptyset$. We remark that $(g,p)S \cap (h'',q'')S = \emptyset$ for all $h'' \in G$ and $q'' \in P_F\setminus\{q\}$. This forces $q \in q'P$, so Lemma~\ref{lem:ideal intersection for gxp} implies $g \in h'\theta_{q'}(G)\theta_q(G) = h'\theta_{q'}(G)$. Since $g$ was arbitrary, we get \eqref{eq:P_F cover of G} for every minimal $q \in P_F$ and hence for all $q \in P_F$. 

The converse direction is straightforward. For each $(g,p) \in S$, there is $q \in P_F$ such that $pP \cap qP \neq \emptyset$. This means $(g,p)S \cap (g,q)S \neq \emptyset$, see Lemma~\ref{lem:ideal intersection for gxp}. By \eqref{eq:P_F cover of G}, there exists $(h',q') \in F$ satisfying $(h',q') \leq (g,q)$. In particular, this implies $(g,p)S \cap (h',q')S \supset (g,p)S \cap (g,q)S \neq \emptyset$, so $F$ is a foundation set for $S$.
\end{proof}

\noindent Note that if $F \subset S$ and $P_F \subset P$ satisfy \eqref{eq:P_F cover of G}, then we have $P_F \subset \bigcup_{(h,q) \in F} qP$.

For the next step, we will need a celebrated lemma of B.H.~Neumann on finiteness properties for covers of groups, see \cite{Neu}*{Lemma 4.1}.

\begin{lemma}\label{lem:Neumann fin index lemma}
Let $G$ be a group and $G_1,\dots,G_n$  subgroups of $G$. If there are $g_1, \dots,g_n \in G$ such that $G= \bigcup_{1 \leq i \leq n} g_iG_i$, then there is $1 \leq i \leq n$ such that the index $[G:G_i] < \infty$ and $G= \bigcup_{\substack{1 \leq i \leq n:\\ [G:G_i] < \infty}} g_iG_i$.
\end{lemma}

\begin{proposition}\label{prop:FS for gxp - pfin suffices}
Let $F$ be a finite subset of $S$. Then $F$ is a foundation set for $S$ if and only if $F \cap G \rtimes_\theta \pfin$ is a foundation set for $S$.
\end{proposition} 
\begin{proof}
If $F$ is a foundation set, then Lemma~\ref{lem:char FS for gxp by P_F} states that there exists $P_F\in \FF(P)$ satisfying \eqref{eq:P_F cover of G} for $F$.
Now if we let $F^{\text{(fin)}}:= F \cap G \rtimes_\theta \pfin$, then Lemma~\ref{lem:Neumann fin index lemma} shows that 
$P_F$ also satisfies \eqref{eq:P_F cover of G} for $F^{\text{(fin)}}$. Hence $F^{\text{(fin)}}$ is a foundation set for $S$ by Lemma~\ref{lem:char FS for gxp by P_F}. The reverse implication is obvious.
\end{proof}

\begin{corollary}\label{cor:no AFS with inf part}
If $F$ is an accurate foundation set, then $F \subset G \rtimes_\theta \pfin$.
\end{corollary}
\begin{proof}
Let $F \in \FF_a(S)$. By Proposition~\ref{prop:FS for gxp - pfin suffices} we know that $F^{\text{(fin)}} := F\cap G \rtimes_\theta \pfin$ is also a foundation set for $S$. So if there was $(g,p) \in F$ with $p \in \pinf$, then there would be $(h,q) \in F^{\text{(fin)}}$ satisfying $(g,p)S \cap (h,q)S \neq \emptyset$. But then $F$ would not be accurate and hence we conclude $F=F^{\text{(fin)}}$.
\end{proof}

\begin{definition}\label{def:elementary AFS for ADS}
If 
\[F = \{ (g_1^{(1)},p_1),\dots,(g_1^{(n_1)},p_1),(g_2^{(1)},p_2),\dots,(g_m^{(n_m)},p_m)\} \subset S\] 
is such that 
\begin{enumerate}[(1)]
\item $\{p_1,\dots,p_m\}$ is contained in $\pfin$ and an element of $\FF_a(P)$, and
\item $(g_\ell^{(k)})_{1 \leq k \leq n_\ell}$ is a transversal for $G/\theta_{p_\ell}(G)$ for each $1 \leq \ell \leq m$,
\end{enumerate}
then $F$ is called an \emph{elementary foundation set}. The collection of all elementary foundation sets is denoted by $\FF_e(\gpt)$.
\end{definition}

\noindent Every elementary foundation set is an accurate foundation set. 

\begin{example}\label{ex:elementary AFS for ADS}
Let us consider $\Z \rtimes |2\rangle \subset \Z \rtimes \Z^\times$ built from the irreversible algebraic dynamical system $(\Z,|2\rangle,\cdot)$. The set $\{(0,2),(1,2)\}$ forms an elementary foundation set whereas $\{(0,2),(1,4),/3,4)\}$ is an accurate foundation set, which is non-elementary.
\end{example}

\begin{proposition}\label{prop:refining FS to AFS for ADS}
Suppose that for every $F \in \FF(P)$ with $F \subset \pfin$ there exists an accurate refinement $F_a \in \FF_a(P)$ with $F_a \subset \pfin$. Then every foundation set for $S$ can be refined accurately by an elementary foundation set for $S$. In particular, $S$ has property \textnormal{(AR)}.
\end{proposition} 
\begin{proof}
Let $F' \in \FF(S)$. Using Proposition~\ref{prop:FS for gxp - pfin suffices}, we may assume $F' \subset G \rtimes_\theta \pfin$. In particular, $F:= \{p \in P \mid (g,p) \in F \text{ for some } g \in G\} \subset \pfin$ forms a foundation set for $P$. By our assumption, there is $F_a \in \FF_a(P)$ with $F_a \subset \pfin$ which refines $F$. Next, pick a transversal $T_p$ for $G/\theta_p(G)$ for every $p \in F_a$. Then $F'_e := \{ (g,p) \mid p \in F_a, g \in T_p \}$ yields an elementary foundation set that refines $F'$. Since elementary foundation set are accurate, $S$ has property (AR).
\end{proof}

\begin{remark}\label{rem:towards a converse for FS to AFS for ADS}
The converse of the first statement in Proposition~\ref{prop:refining FS to AFS for ADS} might be true in some cases, but there is a subtlety we would like to point out: Suppose $S$ has property (AR) and let $F \in \FF(P)$ with $F \subset \pfin$. Choose a transversal $T_p$ for $G/\theta_p(G)$ for every $p \in F$. As $F \subset \pfin$, the set $F':= \{(g,p) \mid p \in F, g \in T_p\}$ is a foundation set for $S$. Thus there exists $F'_a \in \FF_a(S)$ which refines $F'$. By Proposition~\ref{prop:FS for gxp - pfin suffices}, we know that we can assume $F'_a \subset G \rtimes_\theta \pfin$. It follows that $F_a := \{ p \in P \mid (g,p) \in F'_a \text{ for some } g \in G\}$ is a foundation set for $P$. However, this need not imply that $F_a$ is accurate. In fact, this depends on the choice of a suitable $F'_a$. 
\end{remark}

\noindent We note the following consequence of Proposition~\ref{prop:refining FS to AFS for ADS}:

\begin{corollary}\label{cor:AR for some gxp}
$S$ has property \textnormal{(AR)} provided that 
\begin{enumerate}
\item[\textnormal{(1)}] $P$ is directed with respect to $\leq$, or
\item[\textnormal{(2)}] incomparable elements in $P$ have disjoint principal right ideals.
\end{enumerate}
In particular, $S$ has property \textnormal{(AR)} if $P$ satisfies one of the conditions \textnormal{(3)--(6)} from Corollary~\ref{cor:AR for some right LCM sgps}.
\end{corollary} 
\begin{proof}
By Proposition~\ref{prop:AR for some right LCM sgps}, the prerequisites for Proposition~\ref{prop:refining FS to AFS for ADS} are fulfilled and hence $S$ has property (AR). Since each of the conditions (3)--(6) implies (1) or (2), the additional claim is clear, see Corollary~\ref{cor:AR for some right LCM sgps}.
\end{proof}


\section{The boundary quotient $\CQ(G\rtimes_\theta P)$}\label{sec: BQ for ADS}

\noindent Recall from \cite{BLS1} that the authors associated a $C^*$-algebra 
$\CA[\gpt]$ 
to every algebraic dynamical system $(\gpt)$, and showed that it is 
canonically isomorphic to the full semigroup $C^*$-algebra 
$C^*(G\rtimes_\theta P)$. In this section we use the insights gained in Section~\ref{sec:ADS} to give an 
alternative presentation of the 
boundary quotient 
$\CQ(G\rtimes_\theta P)$ provided that $\gxp$ has property 
(AR). For irreversible algebraic 
dynamical systems, we conclude that $\CQ(\gxp)$ is canonically isomorphic to 
the 
algebra $\CO[\gpt]$ from \cite{Sta1}. We also indicate that $\CQ(\gxp)$ 
can be represented on $\ell^2(G)$ in the obvious way if and only if $P$ is 
directed, which raises the question of a natural state space for $\CQ(\gxp)$ 
for 
the case where $P$ is not directed, see Proposition~\ref{prop:rep on G} and 
Remark~\ref{rem:rep on G}. The majority of this Section appears in  
Subsection~\ref{subsec:p.i. and simple O}, in which we 
address the issues of pure infiniteness and simplicity for $\CQ(\gxp)$.

We will again denote $\gxp$ by $S$ within this section.

\subsection{Basic structure}\label{subsec:basic structure} ~\\
\noindent In this short subsection we obtain a dynamic description of $\CQ(\gxp)$ for algebraic dynamical systems $(\gpt)$ with the property that for every $F \in \FF(P)$ with $F \subset \pfin$ there exists an accurate refinement $F_a \in \FF_a(P)$ with $F_a \subset \pfin$. This allows us to identify $\CQ(\gxp)$ as the $C^*$-algbra $\CO[\gpt]$ from \cite{Sta1} for irreversible algebraic dynamical systems. Moreover, we discuss representability of $\CQ(\gxp)$ on $\ell^2(G)$

\begin{proposition}\label{cor:identifying Q and O}
If $(\gpt)$ is an algebraic dynamical system such that for every $F \in \FF(P)$ with $F \subset \pfin$ there exists an accurate refinement $F_a \in \FF_a(P)$ with $F_a \subset \pfin$, then $\CQ(S)$  is the universal $C^*$-algebra generated by a unitary representation $\bar{u}$ of the group $G$ and a representation $\bar{s}$ of the semigroup $P$ by isometries subject to the relations:
\[\begin{array}{crcll}
\textup{(A1)} & \bar{s}_p \bar{u}_g &=& \bar{u}_{\theta_p(g)}\bar{s}_p & \text{ for all }p\in P, g\in G.\vspace*{2mm}\\
\textup{(A2)} & \bar{s}_p^*\bar{u}_g^{\phantom{*}}\bar{s}_q^{\phantom{*}} &=& 
\multicolumn{2}{l}{\begin{cases}
\bar{u}_{k}^{\phantom{*}}\bar{s}_{p'}^{\phantom{*}}\bar{s}_{q'}^*\bar{u}_{\ell}^* &\text{if $pP \cap qP = rP,\, pp'=qq'=r$ and}\\
&g=\theta_p(k)\theta_q(\ell^{-1}) \text{ for some } k,\ell\in G,\\
0& \text{otherwise.}
\end{cases}}\vspace*{2mm}\\
\textup{(O)}& 1 &=& \sum\limits_{(g,p) \in F}{\bar{e}_{g,p}} & \text{ for every } F \in \FF_e(\gpt),
\end{array}\]
where $\bar{e}_{g,p} = \bar{u}_{g}^{\phantom{*}}\bar{s}_{p}^{\phantom{*}}\bar{s}_{p}^{*}\bar{u}_{g}^{*}$.
\end{proposition}

\begin{proof}
By \cite{BLS2}*{Theorem 4.4}, $C^*(S)$ is isomorphic to $\CA[\gpt]$. (A1) and (A2) represent the defining relations for $\CA[\gpt]$, see \cite{BLS2}*{Definition 2.2}. Hence we need to argue that \eqref{eq:Q condition} and (O) are equivalent. Since $S$ has property (AR), relation \eqref{eq:Q condition} is equivalent to \eqref{eq:Q_a condition}, see Proposition~\ref{prop:AR property -> familiar boundary quot pic}. Clearly, this implies (O) as $F \in \FF_e(\gpt)$ is always an accurate foundation set. Now suppose (O) holds and we have $F \in \FF_a(S)$. By Corollary~\ref{cor:no AFS with inf part}, we now that $F \subset G \rtimes_\theta \pfin$.  Hence there exists $F_e \in \FF_e(\gpt)$ refining $F$, see Proposition~\ref{prop:refining FS to AFS for ADS}. This leads to 
\[\begin{array}{l}
1 \geq \sum\limits_{(g,p) \in F} e_{g,p} \geq \sum\limits_{(g,p) \in F_e} e_{g,p} = 1,
\end{array}\]
which establishes \eqref{eq:Q_a condition} using (O).
\end{proof}

\begin{corollary}\label{cor:consistency for IADS}
Suppose $(\gpt)$ is an irreversible algebraic dynamical system. Then $\CO[\gpt]$ is canonically isomorphic to $\CQ(S)$.
\end{corollary}
\begin{proof}
$P$ is a countably generated, free abelian monoid, hence directed, so Corollary~\ref{cor:AR for some gxp} applies and we arrive at the description of $\CQ(S)$ from Corollary~\ref{cor:identifying Q and O}. A comparison of this presentation with \cite{Sta1}*{Definition 3.1} shows that the two $C^*$-algebras are canonically isomorphic since (CNP1) and (CNP 2) correspond to (A1) and (A2), respectively, and (CNP 3) corresponds to (O) because $P$ is directed, see Remark~\ref{rem:AFS for directed right LCM}.
\end{proof}

\noindent The algebra $\CO[\gpt]$ was constructed from the natural representation of $(\gpt)$ on $\ell^2(G)$. Therefore, we would like to discuss this approach for $\CQ(S)$ for algebraic dynamical systems: 

Let $(\xi_g)_{g \in G}$ denote the standard orthonormal basis for $\ell^2(G)$ and $(U_g)_{g \in G}$ the unitary representation of $G$ on $\ell^2(G)$ given by $U_g \xi_h = \xi_{gh}$. Moreover, the map $\xi_h \mapsto \xi_{\theta_p(h)}$ defines an isometry $S_p \in \CL(\ell^2(G))$ for every $p \in P$.

\begin{proposition}\label{prop:rep on G}
For an algebraic dynamical system $(\gpt)$ and $S= \gxp$, the assignment $\bar{u}_g \bar{s}_p \mapsto U_gS_p$ defines a representation $\lambda$ of $\CQ(S)$ on $\ell^2(G)$ if and only if $P$ is directed.
\end{proposition}
\begin{proof}
If $P$ is directed, then $S$ has property (AR) by Corollary~\ref{cor:AR for some gxp} and hence $\CQ(S)$ can be described as in Corollary~\ref{cor:identifying Q and O}. So we need to show that $(U_g)_{g \in G}$ and $(S_p)_{p \in P}$ satisfy (A1), (A2) and (O). (A1) is obvious and (O) is also easy once we observe that $\FF_a(P) \cap \pfin$ is given by $\pfin$, see Remark~\ref{rem:AFS for directed right LCM}. This means that the families in $\FF_a^{\text{(fin)}}(\gpt)$ consist of one element $p \in \pfin$ together with a transversal $T_p$ for $G/\theta_p(G)$. The verification of (A2) is a straightforward calculation that is omitted here. Thus we get a $*$-homomorphism $\lambda: \CQ(S) \to \CL(\ell^2(G))$ with $\bar{u}_g\bar{s}_p \mapsto U_gS_p$. 

Now suppose $P$ is not directed, that is, there are $p,q \in P$ with disjoint principal right ideals. Then (A2) implies $s_p^*s_q = 0$. But $\theta_p(G) \cap \theta_q(G)$ is a subgroup of $G$ and hence $S_p^*S_q\xi_1 = \xi_1$. In particular, we get $S_p^*S_q \neq 0$, so (A2) does not hold for $S_p, U_1$ and $S_q$. 
\end{proof}

\begin{remark}\label{rem:rep on G}
The $C^*$-algebra $\CA[\gpt]\cong C^*(S)$ introduced in \cite{BLS1} is a 
$C^*$-algebraic model 
for the dynamical system $(\gpt)$ based on the state space $\ell^2(S)$ and 
$\CQ(S)$ is derived from this construction as a quotient. Although
$\ell^2(G)$ is arguably a natural state space, we lose this representation 
for $\CQ(S)$ once we leave the realm of actions of directed semigroups $P$. 
It seems that $\ell^2(G)$ can be too small to accommodate a representation of 
a $C^*$-algebraic model for $(\gpt)$ that incorporates relations on the ideal 
structure of $P$. This raises the question whether there is a natural Hilbert 
space associated to $(\gpt)$ on which we can represent $\CQ(S)$. 
\end{remark} 


\subsection{Simplicity and pure infiniteness \`{a} la Starling}\label{subsec:p.i. and simple 
O}	~\\
\noindent The remainder of this section is devoted to applying the results of 
\cite{Star} which we 
recalled in Section~\ref{subsec: inv semigroup approach to BQ} to right LCM 
semigroups 
$S=\gxp$. This yields necessary and sufficient conditions on $(\gpt)$ for 
$\CQ(S)$ to be purely infinite and simple. We show that these conditions 
look quite familiar in the case where $P$ is right cancellative, an extra 
assumption which is satisfied by many interesting examples.

We first address the issue of simplicity, and then discuss pure 
infiniteness starting after Remark~\ref{rem: minimal stuff}. Before we can 
state any 
results, though, we have to do some work on translating conditions 
\eqref{eq:condition H} 
and \eqref{eq:condition EP} from Theorem~\ref{thm:starling simple} into the 
setting of algebraic dynamical systems. 

Recall from Definition~\ref{def:core for right LCM} that the group 
of units in a 
semigroup is always contained in the core. While this inclusion is proper in 
many cases, we will show that we have equality for algebraic dynamical 
systems $(\gpt)$ provided that $\theta_p \in \Aut(G)$ implies $p \in P^*$ for 
all $p \in P$.

\begin{standingassumption}
For the rest of this section we will assume that for $(\gpt)$ an algebraic 
dynamical system we have $\theta_p \in \Aut(G)\Longrightarrow p \in P^*$ for 
all $p \in P$.
\end{standingassumption}

This is a very reasonable 
assumption since the original $P$ can always be replaced by 
the right LCM semigroup $\{\theta_p \mid p \in P\}$.


\begin{proposition}\label{prop:core for regular ADS}
For an algebraic dynamical system $(\gpt)$ we have $S_0=S^*$.
\end{proposition}

Recall from \cite{BLS1}*{Lemma 2.4} that $S^*= G \rtimes_\theta P^*$.

\begin{proof}[Proof of Proposition~\ref{prop:core for regular ADS}]
Let $(g,p),(h,q) \in S$. According to Lemma~\ref{lem:ideal intersection for gxp}, $(g,p)S \cap (h,q)S \neq \emptyset$ holds if and only if $pP \cap qP \neq \emptyset$ and $h \in g\theta_p(G)\theta_q(G)$. Thus, $(g,p) \in S_0$ if and only if $p \in P_0$ and $h \in g\theta_p(G)\theta_q(G)$ for all $h \in G$ and $q \in P$. If we choose $q=p$, then this implies $G = g\theta_p(G)$ as $\theta_p(G)$ is a subgroup of $G$. Hence we get $\theta_p \in \Aut(G)$ as a necessary condition. But this is also sufficient as $h \in g\theta_p(G)\theta_q(G) = G\theta_q(G) = G$. Thus we see that $S_0 = S^*$.
\end{proof}

\begin{remark}\label{rem:condition H}
Recall from Remark~\ref{rem:condition H general} that condition 
\eqref{eq:condition H} always holds for a right cancellative right LCM 
semigroups. We note that for algebraic dynamical systems $(\gpt)$ with 
$S=\gxp$, this is 
equivalent to $P$ having right cancellation. So the non-trivial case for 
\eqref{eq:condition H} is the one where $S$ is not right cancellative. The set 
$S_{s,t} := \{ r \in S \mid  sr=tr \}$ for $s,t \in S$ is a proper right ideal 
in $S$ unless $s=t$, in which case $S_{s,t} = S$. We note that $S_{s,t}$ is a 
left cancellative semigroup that may also be empty. From this perspective, 
\eqref{eq:condition H} is equivalent to 
\begin{equation}\label{eq:condition H'}
\text{The semigroup $S_{s,t}$ has a foundation set for all $s,t \in S$.}\tag{H'} 
\end{equation}
Here, the term foundation set is meant in the sense of Definition~\ref{def:foundation sets}, even though $S_{s,t}$ need not be right LCM.
\end{remark}

\noindent Due to a lack of examples of algebraic dynamical systems with a 
right LCM semigroup $P$ that is not right cancellative, we stop the discussion 
of condition \eqref{eq:condition H} and commence on \eqref{eq:condition 
EP}. 

Recall from Proposition~\ref{prop:core for regular ADS} that $S_0 = G \rtimes_\theta P^*$ and from \eqref{eq: square brackets} the notation $[s,t]$ for an element in the inverse semigroup $I(S)$ corresponding to $v_sv_t^* \in C^*(S)$. 
In particular, we have $[s,s] = 1$ for all $s \in S_0$ and hence $[s,s] \geq 
[s',s']$ for all $s' \in S$. Since $S_0 = G \rtimes_\theta P^*$ is a group, it 
suffices to consider the case $t=1$ for \eqref{eq:condition EP} because $[s,t] = [st^{-1},1]$. Our aim is 
to find a precise dynamic condition on $(\gpt)$ which guarantees that $s=t$ 
holds as soon as there exists some $[tt',tt']$ that is weakly fixed by $[s,t]$ 
with $s,t \in S_0$. 

\begin{notation}\label{not:G of p,h,q}
For $p \in P^*$ and $(h,q) \in S$, we let 
\[\begin{array}{l} G_{p,h,q} := \bigcap\limits_{(k,r) \in S}h\theta_q(k)\theta_{qr}(G)\theta_{pqr}(G)\theta_p(h\theta_q(k))^{-1}. \end{array}\]
\end{notation}

\noindent We can now state the first of our simplicity results.

\begin{theorem}\label{thm:applying Starling's result to reg ADS}
Suppose $(\gpt)$ is an algebraic dynamical system with right cancellative $P$. Let $S = \gxp$. The boundary quotient $\CQ(S)$ is simple if and only if $\CQ(S) \cong C^*_r(\CG_{\text{tight}}(I(S)))$, and
\begin{equation}\label{eq:condition EP'}
\hspace*{-1mm}\begin{array}{l} 
\text{For $(h,q) \in S$ and all $p \in P^*$ with $pqrP \cap qrP \neq \emptyset$ for all $r\in P$,}\\
\text{the set $G_{p,h,q}$ is empty unless $p=1$, in which case $G_{1,h,q} = \{1\}$.}
\end{array}\tag{EP'}
\end{equation}
\end{theorem}
\begin{proof}
We want to show that $[s,t]$ satisfies \eqref{eq:condition EP} for all $s,t \in S_0$ so that Theorem~\ref{thm:starling simple} applies. Recall from Lemma~\ref{prop:core for regular ADS} that $S_0 = G \rtimes_\theta P^*$. Moreover, as remarked before Notation~\ref{not:G of p,h,q}, it suffices to consider $[(g,p),(1,1)]$ with $(g,p) \in S_0$. So let $(g,p) \in S^* = G \rtimes_\theta P^*$. We start by observing that $[(g,p),(1,1)]$ fixes $[(h,q),(h,q)]$ weakly if and only if 
\begin{equation}\label{eq:weakly fixed formula I}
qr P \cap pqr P \neq \emptyset \text{ for all $r \in P$ and } g \in G_{p,h,q}.
\end{equation}
Indeed, $[(g,p),(1,1)]$ fixes $[(h,q),(h,q)]$ weakly if and only if
\[(h,q)(k,r)S \cap (g,p)(h,q)(k,r)S \neq \emptyset \text{ for all } (k,r) \in S.\]
Using Lemma~\ref{lem:ideal intersection for gxp}, we translate this to $qr P \cap pqr P \neq \emptyset$ and 
\[\bigl(h\theta_q(k)\bigr)^{-1}g\theta_p(h)\theta_{pq}(k) \in \theta_{qr}(G)\theta_{pqr}(G) 
\text{ for all } (k,r) \in S.\]
The second equation can be reformulated as $g \in G_{p,h,q}$. Let us note that $G_{p,h,q}$ may be empty in which case $[(g,p),(1,1)]$ cannot fix $[(h,q),(h,q)]$ weakly irrespective of the choice of $g$.

Since $P$ is right cancellative, so is $S$. In view of Remark~\ref{rem:EP for special cases}~(c), we want to use \eqref{eq:weakly fixed formula I} to show that \eqref{eq:condition EP'} is equivalent to: 
\begin{equation}\label{eq:condition EP phantom}
\begin{array}{l}
\text{If $(g,p) \in G \rtimes_\theta P^*$ has the property that $[(g,p),(1,1)]$ fixes some}\\
\text{$[(h,q),(h,q)]$ weakly, then $(g,p)=(1,1)$.}
\end{array} 
\end{equation}
If \eqref{eq:condition EP'} holds, then the only $(g,p)$, for which $[(g,p),(1,1)]$ may fix some $[(h,q),(h,q)]$ weakly, is $(g,p)=(1,1)$. Thus \eqref{eq:condition EP phantom} is valid. Conversely, suppose there is $(h,q) \in S$ and $p \in \{ p' \in P^* \mid p'qrP \cap qrP \neq \emptyset \text{ for all } r\in P\}$ for which either $p \neq 1$ and $G_{p,h,q} \neq \emptyset$ or $p=1$ and there exists $g \in G_{1,h,q}\setminus\{1\}$. In both cases, we get a $(g,p) \in S_0$ such that $[(g,p),(1,1)]$ fixes $[(h,q),(h,q)]$ weakly by \eqref{eq:weakly fixed formula I}, but $(g,p) \neq (1,1)$. So we arrive at a contradiction to \eqref{eq:condition EP phantom} and the proof is complete.
\end{proof}

\noindent The condition \eqref{eq:condition EP'} is technical and lacks an immediate interpretation. But we will see that it takes a simpler form in special cases. 

\begin{corollary}\label{cor:simple Q(S) for reg ADS with P*p in pP*}
Suppose $(\gpt)$ is an algebraic dynamical system with right cancellative $P$ and $P^*p \subset pP^*$ for all $p \in P$. Let $S = \gxp$. The boundary quotient $\CQ(S)$ is simple if and only if 
\begin{enumerate}
\item[\textnormal{(1)}] $\CQ(S) \cong C^*_r(\CG_{\text{tight}}(I(S)))$, and
\item[\textnormal{(2)}] $\bigcap_{(k,r) \in S} k\theta_r(G)k^{-1} = \{1\}$, and 
\item[\textnormal{(3)}] $\bigcap_{(k,r) \in S} k\theta_r(G)\theta_{\tilde{p}}(k)^{-1} = \emptyset$ for all $\tilde{p} \in P^*$ arising from $pq=q\tilde{p}$ for some $p \in P^*\setminus\{1\} \text{ and } q \in P$.
\end{enumerate}
\end{corollary}
\begin{proof}
We claim that \eqref{eq:condition EP'} holds if and only if (2) and (3) are true. Let $p \in P^*$ and $(h,q) \in S$. We start by observing that \eqref{eq:weakly fixed formula I} holds for all $p\in P^*$ as $pqrP = qrP$ for all $q,r \in P$. Moreover, writing $pq=q\tilde{p}$ with $\tilde{p} \in P^*$, the set $G_{p,h,q}$ becomes
\[\begin{array}{lcl}
G_{p,h,q} &=&  \bigcap\limits_{(k,r) \in S}h\theta_q(k)\theta_{qr}(G)\theta_{pqr}(G)\theta_p(h\theta_q(k))^{-1} \\
&\stackrel{a)}{=}&  \bigcap\limits_{(k,r) \in S}h\theta_q(k)\theta_{qr}(G)\theta_{q\tilde{p}}(k)^{-1}\theta_p(h)^{-1}\\
&\stackrel{b)}{=}&  h\theta_q\bigl(\bigcap\limits_{(k,r) \in S}k\theta_r(G)\theta_{\tilde{p}}(k)^{-1}\bigr)\theta_p(h)^{-1},
\end{array}\]
where we used that:
\begin{enumerate}[a)]
\item $\theta_{pqr}(G) = \theta_{qr}(G)$ for $pqr=qrp'$ and $\theta_{p'} \in \Aut(G)$.
\item $G$ is a group and $\theta_q$ is injective.
\end{enumerate}
This proves the claim and hence we can apply Theorem~\ref{thm:applying Starling's result to reg ADS}.
\end{proof}

\begin{remark}\label{rem:absorbing q for P*} 
The existence of $p \in P^*$ and $q \in P$ such that $p \neq 1$, but $pq=q$, i.e. $\tilde{p}=1$, immediately leads to a violation of \eqref{eq:condition EP'} as $\bigcap_{(k,r) \in S} k\theta_r(G)k^{-1}$ is a subgroup of $G$. Note that this phenomenon can only occur in the case where $P$ is not right cancellative.
\end{remark}

\begin{remark}\label{rem:simple Q(S) for reg ADS with P*p in pP*}
Suppose that $P$ is right cancellative. If $\theta$ separates the points of $G$, i.e. $\bigcap_{p \in P} \theta_p(G) = \{1\}$, and $\theta:P^* \curvearrowright G$ is faithful, that is, for each $p \in P^*\setminus \{1\}$ there exists $g \in G$ with $\theta_p(g) \neq g$, then conditions (2) and (3) from Corollary~\ref{cor:simple Q(S) for reg ADS with P*p in pP*} are satisfied. Indeed, (2) is obvious. If we take $p \in P^*\setminus\{1\} \text{ and } q \in P$ to get $\tilde{p} \in P^*$ with $pq=q\tilde{p}$, right cancellation for $P$ implies $\tilde{p} \neq 1$. Since $\theta:P^* \curvearrowright G$ is faithful, there is $g \in G$ with $\theta_{\tilde{p}}(g) \neq g$. If $\theta$ separates the points in $G$, we can choose $r' \in P$ large enough such that $g^{-1}\theta_{\tilde{p}}(g) \notin \theta_{r'}(G)$. Therefore, 
\[\begin{array}{c} \bigcap\limits_{(k,r) \in S} k\theta_r(G)\theta_{\tilde{p}}(k)^{-1} \subset \bigcap\limits_{r \in P} \theta_r(G) \cap g\theta_{r'}(G)\theta_{\tilde{p}}(g)^{-1} = \emptyset, \end{array}\]
which shows (3).
\end{remark}

\noindent For $P^*= \{1\}$, we recover a condition that has already appeared in \cite{BLS1}. Recall that an action $H \curvearrowright J$ of a group $H$ on a set $J$ is said to be \emph{effective} if for every $h \neq 1$ there is $X \in J$ such that $h.X \neq X$. 

\begin{corollary}\label{cor:applying Starling's result to reg ADS with trivial P*}
Suppose $(\gpt)$ is an algebraic dynamical system with $P^*=\{1\}$ and right cancellative $P$. Let $S = \gxp$. The boundary quotient $\CQ(S)$ is simple if and only if 
\begin{enumerate}
\item[\textnormal{(1)}] $\CQ(S) \cong C^*_r(\CG_{\text{tight}}(I(S)))$, and
\item[\textnormal{(2)}] $S^* \curvearrowright \CJ(S)$ is effective.
\end{enumerate}
\end{corollary}
\begin{proof}
In the case of $P^*=\{1\}$, \cite{BLS1}*{Lemma 8.5 and Lemma 8.6} states that the action $S^* \curvearrowright \CJ(S)$ for $S= \gxp$ is effective if and only if $\bigcap_{(k,r) \in S} k\theta_r(G)k^{-1} = \{1\}$. Now Corollary~\ref{cor:simple Q(S) for reg ADS with P*p in pP*} applies because condition (3) is void due to $P^* = \{1\}$.
\end{proof}

\noindent Corollary~\ref{cor:applying Starling's result to reg ADS with trivial P*} yields a sophisticated answer to the question of a characterisation of simplicity of $\CO[\gpt]$ for irreversible algebraic dynamical systems $(\gpt)$ considered in \cite{Sta1}, where sufficient conditions were discussed, see \cite{Sta1}*{Theorem 3.26}. Due to \cite{Sta1}*{Definition 1.5 (C)}, $\theta_p \in \Aut(G)$ implies $p=1 \in P^*$. Moreover, $P$ is right cancellative and $P^* = \{1\}$ since $P$ is a countably generated, free abelian monoid. 

\begin{corollary}\label{cor:simplicity for IADS}
Let $(\gpt)$ be an irreversible algebraic dynamical system. Then $\CO[\gpt]$ is simple if and only if 
\begin{enumerate}
\item[\textnormal{(1)}] $\CO[\gpt] \cong C^*_r(\CG_{\text{tight}}(I(\gxp)))$, and
\item[\textnormal{(2)}] $\bigcap\limits_{(g,p) \in \gxp} g\theta_p(G)g^{-1} = \{1\}$.
\end{enumerate}
\end{corollary}
\begin{proof}
By Corollary~\ref{cor:consistency for IADS}, we have $\CQ(\gxp) \cong \CO[\gpt]$. As $P$ is right cancellative and $P^*=\{1\}$, the claim follows from Corollary~\ref{cor:applying Starling's result to reg ADS with trivial P*}.
\end{proof}

\begin{remark}\label{rem: minimal stuff}
In \cite{Sta1}*{Theorem 3.26}, the author proved that $\CO[\gpt]$ is simple 
(and purely infinite) given that the canonical action $\hat{\tau}$ of $S^* 
\cong G$ on the spectrum $G_{\theta}$ of the diagonal of $\CO[\gpt]$ is 
amenable and that the action $\theta$ is \emph{minimal} in the sense that 
$\bigcap_{p \in P} \theta_p(G) = \{1\}$. It is not hard to see that 
amenability of $\hat{\tau}$ yields amenability of 
$\CG_{\text{tight}}(I(\gxp))$ and hence $\CO[\gpt] \cong 
C^*(\CG_{\text{tight}}(I(\gxp))) \cong C^*_r(\CG_{\text{tight}}(I(\gxp)))$. In 
addition, minimality of $(\gpt)$ clearly implies (2) from 
Corollary~\ref{cor:simplicity for IADS}. So we see that, in general, the 
conditions on $(\gpt)$ are slightly milder than the ones obtained in 
\cite{Sta1}. Note however, that minimality of $(\gpt)$ is in fact necessary 
and sufficient for simplicity of $\CO[\gpt]$ in the case where $G$ is abelian, 
as assumed in \cite{CV}.\vspace*{3mm}
\end{remark}

\noindent Let us now briefly discuss pure infiniteness of $\CQ(S)$ for $S=\gxp$. It was proven in \cite{Star}*{Theorem 4.15} that, for general right LCM $T$, the boundary quotient $\CQ(T)$ is purely infinite if and only if $\CG_{\text{tight}}(T)$ is not a single point, provided that $\CQ(T)$ is simple and $T$ satisfies condition \eqref{eq:condition H}. Hence, pure infiniteness is almost automatic in this case. Indeed, $\CG_{\text{tight}}(T)$, as a set, is given by the equivalence classes of 
\[\begin{array}{ll}
\CG'_{\text{tight}}(T) \vspace*{2mm}\\
&\hspace*{-15mm} := \{\bigl([s,t],\xi\bigr) \mid s,t \in T, \xi \subset \CJ(T) \text{ tight filter with } r \in tT \text{ for all } rT \in \xi\}
\end{array}\]
with respect to $\bigl([s,t],\xi\bigr) \sim \bigl([s',t'],\xi'\bigr)$ defined as 
\[\xi=\xi' \text{ and there exists } rT \in \xi \text{ such that } [s,t].rT = [s',t'].rT,\] 
where $[s,t].rT := s(t^{-1}(rT))$, see \cite{Star}*{Subsection 4.1} for details. 



\begin{corollary}\label{cor:pure inf if not a group}
Suppose that $(\gpt)$ is an algebraic dynamical system such that \eqref{eq:condition H} holds for $S$ and $\CQ(S)$ is simple. Then $\CQ(S)$ is purely infinite if and only if $P$ is not a group.  
\end{corollary}
\begin{proof}
We start by observing that $\CJ(S)= \{\emptyset,S\}$ holds if and only if $S$ is a group (as $sS=S$ implies $s \in S^*$ for all $s \in S$). In this case, the only tight filter on $\CJ(S)$ is $\{S\}$ and every $[s,t] \in I(S)$ fixes $S$, so $\CG_{\text{tight}}(I(S))$ is just a singleton. So if $S$ is a group, which is equivalent to $P$ being a group, then $\CQ(S) \cong \IC$. If $P$ is not a group, then there is $p \in P$ such that $[G:\theta_p(G)] \geq 2$. So there are $g_1,g_2 \in G$ with $g_1^{-1}g_2 \notin \theta_p(G)$. This amounts to $(g_1,p)S \cap (g_2,p)S = \emptyset$ using Lemma~\ref{lem:ideal intersection for gxp}. There is at least one ultrafilter $\xi_i$ of $\CJ(S)$ containing $(g_i,p)S$ for $i=1,2$. We clearly have $(g_2,p)S \notin \xi_1$ as $(g_1,p)S \cap (g_2,p)S = \emptyset$, so $\xi_1 \neq \xi_2$. Therefore, $\CG_{\text{tight}}(S)$ contains at least two distinct points and, with the help of \cite{Star}*{Theorem 4.15}, we conclude that $\CQ(S)$ is purely infinite.
\end{proof}

\noindent The essential ingredient in here is that $S$ fails to be left reversible as soon as $P$ is not a group. It was pointed out to us by Xin Li that there is a deeper connection between pure infiniteness of $C^*$-algebras associated to left cancellative semigroups (without assuming simplicity) and failure with respect to left reversibility.

As a final result, we collect a number of cases in which we now know that $\CQ(\gxp)$ belongs to a well-understood class of $C^*$-algebras classifiable by K-theory, see \cites{Kir,Phi}.

\begin{theorem}\label{thm:UCT Kirchberg algs}
Suppose $(\gpt)$ is an algebraic dynamical system such that $P$ is not a group and $\CG_{\text{tight}}(I(\gxp))$ is amenable. Then $\CQ(\gxp)$ is a unital UCT Kirchberg algebra provided that one of the following conditions holds:
\begin{enumerate}
\item[\textnormal{(1)}] $\gxp$ satisfies \eqref{eq:condition H} and for all $s,t \in (\gxp)_0$, the element $[s,t]$ satisfies \eqref{eq:condition EP}.
\item[\textnormal{(2)}] $P$ is right cancellative, $P^*p \subset pP^*$ for all $p \in P$, and $\gxp$ satisfies \textnormal{(2)} and \textnormal{(3)} from Corollary~\ref{cor:simple Q(S) for reg ADS with P*p in pP*}.
\item[\textnormal{(3)}] $P$ is right cancellative, $P^*=\{1\}$ and $\bigcap_{(g,p) \in \gxp} g\theta_p(G)g^{-1} = \{1\}$.
\end{enumerate}
\end{theorem}
\begin{proof}
In each case, $\CQ(\gxp)$ is simple, see Theorem~\ref{thm:starling simple}, Theorem~\ref{thm:applying Starling's result to reg ADS}, Corollary~\ref{cor:simple Q(S) for reg ADS with P*p in pP*}, and Corollary~\ref{cor:applying Starling's result to reg ADS with trivial P*}, respectively. Since $P$ is not a group, Corollary~\ref{cor:pure inf if not a group} shows that $\CQ(\gxp)$ is also purely infinite. $\CQ(\gxp)$ is separable since $\gxp$ is countable. Finally, nuclearity and the UCT follow from amenability of $\CG_{\text{tight}}(I(\gxp))$, see for instance \cite{BO}*{Theorem 5.6.18} and \cite{Tu}.
\end{proof}

\noindent We remark that Theorem~\ref{thm:UCT Kirchberg algs} generalises \cite{Sta1}*{Corollary 3.28}.

\end{document}